\documentclass[a4paper,12pt]{article}

\usepackage{amssymb}
\usepackage{epsfig}
\usepackage{amsfonts}
\usepackage{amsmath}
\usepackage{euscript}
\usepackage{amscd}
\usepackage{amsthm}
\DeclareMathAlphabet{\mathpzc}{OT1}{pzc}{m}{it}
\newtheorem{thm}{Theorem}[section]
\newtheorem{lem}[thm]{Lemma}
\newtheorem{prop}[thm]{Proposition} 
\newtheorem{defn}[thm]{Definition}

\newtheorem{rem}[thm]{Remark}

\newtheorem{ques}[thm]{Question}
\newcommand{\p}{\mathpzc{p}}
\newcommand{\m}{\mathpzc{m}}
\newcommand{\n}{\mathpzc{n}}

\newcommand{\bQ}{\mathbb Q}

\newcommand{\A}{\mathbb A}

\title{On Residual and Stable Coordinates}
\author{Amartya Kumar Dutta{\footnote
{{\small{\it Statistics and Mathematics  Unit, Indian Statistical Institute,}}
{\small{\it 203 B.T. Road, Kolkata 700 108, India.}}
{\small{\it e-mail : amartya.28@gmail.com}}}},
Animesh Lahiri{\footnote{{\small {\it Swami Vivekananda Research Centre,
Ramakrishna Mission Vidyamandira,}}
{\small {\it P.O. Belur Math, Howrah 711202, India.}}
{\small {\it e-mail : 255alahiri@gmail.com }}}}}  

\begin{document}
\date{}
\maketitle

\begin{abstract}
In a recent paper \cite{K}, M. E. Kahoui and M. Ouali have proved that over an algebraically closed field $k$ of characteristic zero, residual coordinates in $k[X][Z_1,\dots,Z_n]$ are one-stable coordinates. In this paper we extend their result to the case of an algebraically closed field $k$ of arbitrary characteristic. In fact, we show that the result holds when $k[X]$ is replaced by any one-dimensional seminormal domain $R$ which is affine over an algebraically closed field $k$. For our proof, we extend a result of S. Maubach in \cite{M} giving a criterion for a polynomial of the form $a(X)W+P(X,Z_1,\dots,Z_n)$ to be a coordinate in $k[X][Z_1,\dots,Z_n,W]$.\\
\indent 
Kahoui and Ouali had also shown that over a  Noetherian $d$-dimensional ring $R$ containing $\bQ$ any residual coordinate in $R[Z_1,\dots,Z_n]$  is an $r$-stable coordinate, where $r=(2^d-1)n$. We will give a sharper bound for $r$ when $R$ is affine over an algebraically closed field of characteristic zero.   
\smallskip
  

\end{abstract}

\section{Introduction}
We will assume all rings to be commutative containing unity. The notation  $R^{[n]}$ will be used to denote any $R$-algebra isomorphic to a polynomial algebra in $n$ variables over $R$.
 
\medskip
We will discuss results connecting coordinates, residual coordinates and stable coordinates in polynomial algebras (see 2.2 to 2.4 for definitions). The study was initiated for the case $n=2$ by Bhatwadekar and Dutta (\cite{BD1}) and later extended to $n>2$ by  Das and Dutta (\cite{DD}).

\medskip
An important problem in the study of polynomial algebras is to find fibre conditions for a polynomial $F$ in a polynomial ring $A=R^{[n+1]}$ to be a coordinate in $A$. In the case $R=k[X]$, where $k$ is an algebraically closed field of characteristic zero,  S. Maubach gave  the following useful result for polynomials linear in one of the variables (\cite[Theorem 4.5]{M}).
\begin{thm}\label{maubach}
Let $k$ be an algebraically closed field of characteristic zero, $P(X,Z_1,\dots,Z_n)$ be  an element in the polynomial ring $k[X,Z_1,\dots,Z_n]$ and $a(X)\in{k[X]}-\{0\}$. Suppose, for each root $\alpha$ of $a(X)$, $P(\alpha,Z_1,\dots,Z_n)$ is a coordinate in $k[Z_1,\dots,Z_n]$. Then, the polynomial $F$ defined by $F:=a(X)W+P(X,Z_1,\dots,Z_n)$ is a coordinate in $k[X,Z_1,\dots,Z_n,W](=k^{[n+2]})$, along with $X$. 
\end{thm}
In this paper, we will show that Maubach's result holds when $k[X]$ is replaced by \textit{any} one-dimensional ring $R$ which is affine over an algebraically closed field $k$ such that either the characteristic of $k$ is zero or $R_{red}$ is seminormal and  $a(X)$ is replaced by a non-zerodivisor $a$ in $R$ for which the image of $P$ becomes a coordinate over $R/aR$ (see Theorem \ref{ext-Maubach}).

\medskip
As an application of Theorem \ref{maubach}, Kahoui and Ouali  have recently proved the following two results  on the connection between residual coordinates and stable coordinates (\cite[Theorem 1.1 and Theorem 1.2]{K}).
\begin{thm}\label{Kahoui}
Let $k$ be an algebraically closed field of characteristic zero, $R=k[X]~(=k^{[1]})$ and $A=R[Z_1,\dots,Z_n]~(=R^{[n]})$. Then every residual coordinate in $A$ is a $1$-stable coordinate in $A$.
\end{thm}
\begin{thm}\label{Kahoui-general}
Let $R$ be a Noetherian $d$-dimensional ring containing $\bQ$ and $A=R[Z_1,\dots,Z_n]~(=R^{[n]})$. Then every residual coordinate in $A$ is a $(2^d-1)n$-stable coordinate in $A$.
\end{thm}
Using our generalization of Maubach's result (Theorem \ref{ext-Maubach}) and the concept  of exponential maps (see Definition \ref{exp map} and Proposition \ref{exp-auto}) we will generalize Theorem \ref{Kahoui} to the case when $k[X]$ is replaced by a one-dimensional ring $R$ which is affine over an algebraically closed field $k$ such that either the characteristic of $k$ is zero or $R_{red}$ is seminormal (see Theorem \ref{ext-Kahoui}). Next, we will show (Theorem \ref{ext kahoui general}) that the condition ``$R$ contains $\bQ$'' can be dropped from Theorem \ref{Kahoui-general}. We will also show that when $R$ is affine over an algebraically closed field of characteristic zero, the bound $(2^d-1)n$ given in Theorem \ref{Kahoui-general} can be reduced to $2^{d-1}(n+1)-n$ (see Theorem \ref{efficient bound}).

\section{Preliminaries}
In this section we recall a few definitions and some well-known results. 
\begin{defn}
{\em A reduced ring $R$ is said to be {\it seminormal} if it satisfies the  condition: for $b,c\in{R}$ with $b^3=c^2$, there is an $a\in{R}$ such that $a^2=b$ and $a^3=c$.}
\end{defn}
\begin{defn}
{\em Let $A=R[X_1,\dots,X_n]~(=R^{[n]})$ and $F\in{A}$. $F$ is said to be a {\it coordinate in $A$}  if there exist $F_2,\dots,F_n \in{A}$ such that $A=R[F,F_2,\dots,F_n]$.}
\end{defn}
\begin{defn}
{\em Let $A=R[X_1,\dots,X_n]~(=R^{[n]})$, $F\in{A}$ and $m\geq{0}$. $F$ is said to be an {\it $m$-stable coordinate in $A$} if $F$ is a coordinate in $A^{[m]}$.}
\end{defn}
\begin{defn}
{\em Let $A=R[X_1,\dots,X_n]~(=R^{[n]})$ and $F\in{A}$. $F$ is said to be a {\it residual coordinate in $A$} if, for each prime ideal $\p$ of $R$, $A\otimes_{R}k(\p)=k(\p)[\overline{F}]^{[n-1]}$, where $\overline{F}$ denotes the image of $F$ in $A\otimes_{R}k(\p)$ and $k(\p):= \dfrac{R_{\p}}{\p R_{\p}}$ is the residue field of $R$ at $\p$.}
\end{defn}
Now, we state some known results connecting coordinates, residual coordinates and stable coordinates in polynomial algebras. First, we state an elementary result (\cite[Lemma 1.1.9]{E}).
\begin{lem}\label{nilradical}
Let $R$ be a ring, $nil(R)$ denote the nilradical of $R$ and $F\in{R[Z_1,\dots,Z_n]}~(=R^{[n]})$. Let $\overline{R}:=\dfrac{R}{nil(R)}$ and $\overline{F}$ denote the image of $F$ in $\overline{R}[Z_1,\dots,Z_n]$. Then $F$ is a coordinate in $R[Z_1,\dots,Z_n]$ if and only if $\overline{F}$ is a coordinate in $\overline{R}[Z_1,\dots,Z_n]$. 
\end{lem}
The following result has been proved by Kahoui and Ouali in \cite[Lemma 3.4]{K}.
\begin{prop}\label{Artinian}
Let $R$ be an Artinian ring and $A=R[Z_1,\dots,Z_n]~(=R^{[n]})$. Then every residual coordinate in $A$ is a coordinate in $A$.
\end{prop}
 The following result on residual coordinates was proved by Bhatwadekar and Dutta (\cite[Theorem 3.2]{BD1}). 
\begin{thm}\label{res-stable}
Let $R$ be a Noetherian ring such that either $R$ contains $\bQ$ or $R_{red}$ is seminormal. Let $A= R^{[2]}$ and $F\in{A}$. 
If $F$ is a residual coordinate in $A$, then $F$ is a coordinate in $A$.
\end{thm}
Now, we state a result on stable coordinates due to J. Berson, J. W. Bikker and A. van den Essen (\cite[Proposition 5.3]{BBE}); the following version was observed by Kahoui and Ouali in \cite{K}.
\begin{thm}\label{essen stable}
Let $R$ be a ring, $a$ be a non-zerodivisor of $R$ and $P\in{R[Z_1,\dots,Z_n]}~(=R^{[n]})$. Suppose, the image of $P$ is an $m$-stable coordinate in $\dfrac{R}{aR}[Z_1,\dots,Z_n]$. Then the polynomial $F$ defined by $F:=aW+P$ is a $(2m+n-1)$-stable coordinate in  $R[Z_{1},\dots, Z_{n},W]~(=R^{[n+1]})$.
\end{thm}
The following result on linear planes over a discrete valuation ring was proved by S.M. Bhatwadekar and A.K. Dutta in \cite[Theorem 3.5]{BhA}.
\begin{thm}\label{dvr}
Let $R$ be a discrete valuation ring with  parameter $\pi$, $K=R[\frac{1}{\pi}]$, $k=\frac{R}{\pi R}$ and $F=aW-b\in{R[Y,Z,W]}~(=R^{[3]})$, where $a(\neq{0}),b\in{R[Y,Z]}$. Suppose that $\dfrac{R[Y,Z,W]}{(F)}=R^{[2]}$. Let, for each $G\in{R[Y,Z,W]}$, $\overline{G}$ denote the image of $G$ in $k[Y,Z,W]$. Then there exists an element $Y_0\in{R[Y,Z]}$ such that $a\in{R[Y_0]}$, $\overline{Y_0}\notin{k}$ and $K[Y,Z]=K[Y_0]^{[1]}$. Moreover, if $dim(k[\overline{F},\overline{Y_0}])=2$, then $F$ is a coordinate in $R[Y,Z,W]$.
\end{thm}
Now, we define $\A^{r}$-fibration and state a theorem of A. Sathaye (\cite[Theorem 1]{S}) on the triviality of $\A^2$-fibration over a discrete valuation ring containing $\bQ$.
\begin{defn}
{\em An $R$-algebra $A$ is said to be an {\it $\A^{r}$-fibration over $R$} if the following hold:
\begin{enumerate}
\item[\rm(i)] $A$ is finitely generated over $R$,
\item[\rm(ii)] $A$ is flat over $R$, 
\item[\rm(iii)] $A\otimes_{R}{k(\p)}=k(\p)^{[r]}$, for each prime ideal $\p$ of $R$.
\end{enumerate}}
\end{defn}

\begin{thm}\label{Sathaye}
Let $R$ be a discrete valuation ring containing $\bQ$. Let $A$ be an $\A^2$-fibration over $R$. Then $A=R^{[2]}$.
\end{thm}
Next, we define exponential maps and record a basic result.
\begin{defn}\label{exp map}
{\em Let $k$ be a field of arbitrary characteristic, $R$ a $k$-algebra and $A$ be an $R$-algebra. Let $\delta:A\longrightarrow{A^{[1]}}$ be an $R$-algebra homomorphism. We write 
$\delta=\delta_{W}:A\longrightarrow{A[W]}$ if we wish to emphasize an indeterminate $W$. We say $\delta$ is an $R$-linear {\it exponential map} if
\begin{enumerate}
\item[\rm(i)] ${\epsilon_{0}}{\delta_{W}}$ is the identity map on $A$, where ${\epsilon_{0}}:A[W]\longrightarrow{A}$ is the $A$-algebra homomorphism defined by $\epsilon_{0}(W)=0$.
\item[\rm(ii)] ${\delta_{U}}{\delta_{W}}=\delta_{U+W}$, where $\delta_{U}$ is extended to a homomorphism of
$A[W]$ into $A[U,W]$ by setting ${\delta_{U}}(W)=W$.
\end{enumerate}}
\end{defn}

\noindent
For an exponential map $\delta : A\longrightarrow{A^{[1]}}$, we get a sequence of maps ${\delta}^{(i)}:A\longrightarrow{A}$ as follows: for    
$a\in{A}$, set  ${\delta}^{(i)}(a)$ to be the coefficient of $W^i$ in $\delta_{W}{(a)}$, i.e., for $\delta_{W}:A\longrightarrow{A[W]}$, we have   
$$\delta_{W}{(a)}=\sum{\delta^{(i)}(a)W^{i}}.$$
Note that since $\delta_{W}(a)$ is an element in $A[W]$, the sequence ${\lbrace\delta^{(i)}(a)\rbrace}_{i\geq{0}}$ has only finitely many nonzero elements for each $a\in{A}$. Since $\delta_{W}$ is a ring homomorphism, we see that ${\delta}^{(i)}:A\longrightarrow{A}$ is linear for each $i$ and that the Leibnitz Rule:
 ${\delta}^{(n)}(ab)=\underset{i+j=n}{\sum}{\delta^{(i)}(a)\delta^{(j)}}(b)$ holds for all $n$ and for all $a,b\in{A}$. 

\medskip
\noindent 
 The above properties (i) and (ii) of the exponential map $\delta_W$ translate into the following
 properties: 
 \begin{enumerate}
 \item[\rm(i)$^{\prime}$] $\delta^{(0)}$ is the identity map on $A$.
 \item[\rm(ii)$^{\prime}$] The ``iterative property'' $\delta^{(i)}\delta^{(j)}={\binom{i+j}{j}}\delta^{(i+j)}$ holds for all $i,j\geq{0}$.
\end{enumerate}
The following result on exponential maps can be deduced from properties (ii) and (i)$'$ stated above.
\begin{prop}\label{exp-auto}
Let $k$ be a field of arbitrary characteristic, $R$ a $k$-algebra and $A$ an $R$-algebra. Let $\delta_{W}:A\longrightarrow {A[W]}$ be an $R$-linear exponential map and  ${\lbrace\delta^{(i)}\rbrace}_{i\geq{0}}$ the sequence of maps on $A$ defined above. Then the extension of $\delta_{W}$ to $\widetilde{\delta_W}:A[W]\longrightarrow {A[W]}$ defined by setting $\widetilde{\delta_W}(W)=W$, is an $R[W]$-automorphism of $A[W]$ with inverse given by the sequence ${\lbrace(-1)^{i}\delta^{(i)}\rbrace}_{i\geq{0}}$.
\end{prop}
Next, we quote some famous results which will be needed later in this paper. First, we state Bass's cancellation theorem (\cite[Theorem 9.3]{B}).
\begin{thm}\label{Bass Cancellation}
Let $R$ be a Noetherian $d$-dimensional ring and $Q$ be a finitely generated projective $R$-module whose rank at each localization at a prime ideal is at least $d+1$. Let $M$ be a finitely generated projective $R$-module such that $M\oplus{Q}\cong{M\oplus{N}}$ for some $R$-module $N$. Then $Q\cong{N}$.  
\end{thm}
Now, we state Quillen's local-global theorem (\cite[Theorem 1]{Q}).
\begin{thm}\label{extended}
Let $R$ be a Noetherian ring, $D=R^{[1]}$ and $M$ be a finitely generated $D$-module. Suppose, for each maximal ideal $\m$ of $R$, $M_{\m}$ is extended from $R_{\m}$. Then $M$ is extended from $R$. 
\end{thm}
For convenience, we record a basic result on symmetric algebras of finitely generated modules (\cite[Lemma 1.3]{Ea}).
\begin{lem}\label{symm}
Let $R$ be a ring and $M,N$ two finitely generated $R$-modules. If $Sym_{R}(M)$ and ${Sym_{R}(N)}$ denote the respective symmetric algebras then $Sym_{R}(M)\cong{Sym_{R}(N)}$ as $R$-algebras if and only if $M\cong{N}$ as $R$-modules.
\end{lem}
Finally, we quote the theorem on the triviality of locally polynomial algebras proved by Bass-Connell-Wright (\cite{BCW}) and independently by Suslin (\cite{Su}).
\begin{thm}\label{eq: local-global}
Let $A$ be a finitely presented $R$-algebra. Suppose that for each maximal ideal $\m$ of $R$, the $R_{\m}$-algebra $A_{\m}$ is $R_{\m}$-isomorphic to the symmetric algebra of some $R_{\m}$-module. Then $A$ is $R$-isomorphic to the symmetric algebra of a finitely generated $projective$ $R$-module.
\end{thm}
\section{Main Results}
In this section we prove our main results. First, we will extend Theorem \ref{maubach} 
(see Theorem \ref{ext-Maubach}).
For convenience, we record below a local-global result.
\begin{lem}\label{one dimensional cancellation}
Let $R$ be a one-dimensional ring, $n\geq{2}$, $A=R[Z_1,\dots,Z_n]~(=R^{[n]})$ and $F\in{A}$. 
Suppose that for each maximal ideal $\m$ of $R$, $A_{\m}=R_{\m}[F]^{[n-1]}$. Then $F$ is a coordinate in $A$.
\end{lem}   
\begin{proof}
Let $D:=R[F]$, $\n$ be an arbitrary maximal ideal of $D$, $\p:=\n\cap{R}$ and $\m_0$ a maximal ideal of $R$ such that $\p\subseteq{\m_0}$. From the natural maps $R_{\m_{0}}\longrightarrow{R_{\p}}\longrightarrow{D_{\n}}$, we see that $A_{\p}=R_{\p}[F]^{[n-1]}$ and hence $A_{\n}={D_{\n}}^{[n-1]}$. By Theorem \ref{eq: local-global}, there exists a projective $D$-module $Q'$ of rank $(n-1)$  such that $A\cong{Sym_{D}(Q')}$. Since, for each maximal ideal $\m$ of $R$, $A_{\m}=R_{\m}[F]^{[n-1]} \cong Sym_{D_{\m}}(({D_{\m})}^{n-1})$, by Lemma 2.16, 
 we have $Q'_{\m} \cong ({D_{\m}})^{n-1} \cong  (R_{\m})^{n-1} \otimes_R D $. Thus,    $Q'$ is locally extended from $R$ and hence by Theorem \ref{extended}, $Q'$ is extended from $R$, i.e., there exists a projective $R$-module $Q$ of rank $(n-1)$ such that $Q'=Q\otimes_{R}{D}$. Therefore, $A\cong{{Sym_{R}(Q)}\otimes_{R}{D}}\cong{Sym_{R}(Q)\otimes_{R}{Sym_{R}(R)}}\cong{Sym_{R}(Q\oplus{R})}$. Since $A=R[Z_1,\dots,Z_n]\cong{Sym_{R}(R^{n})}$, by Lemma \ref{symm}, we have $Q\oplus{R}\cong{R^{n}}$. Hence, by Theorem \ref{Bass Cancellation}, $Q$ is a free $R$-module of rank $(n-1)$. Therefore, $Q'$ is a free $D$-module of rank $(n-1)$. Hence, $A=R[F]^{[n-1]}$.
\end{proof}
We now extend Theorem \ref{maubach}.
\begin{thm}\label{ext-Maubach}
Let $k$ be an algebraically closed field and $R$ a one-dimensional affine $k$-algebra. Let $a$ be a non-zerodivisor in $R$ and $P(Z_{1},\dots, Z_{n})\in{R[Z_{1},\dots, Z_{n}]}~(=R^{[n]}$) be such that the image of $P$ is a coordinate in $\dfrac{R}{aR}[Z_{1},\dots, Z_{n}]$. If $R_{red}$ is seminormal or if the characteristic of $k$ is zero  then the polynomial $F$ defined by $F:=aW+P$ is a coordinate in  $R[Z_{1},\dots, Z_{n},W]~(=R^{[n+1]})$.
\end{thm}
\begin{proof}
By Lemma \ref{one dimensional cancellation}, it is enough to consider the case when $R$ is a local ring with unique maximal ideal $\m$. Since $R$ is an affine algebra over the algebraically closed field $k$, the residue field $\dfrac{R}{\m}$ is $k$ (by Hilbert's Nullstellensatz). Let $\eta$ denote the canonical map $R[Z_1,\dots,Z_n,W]\longrightarrow{k[Z_1,\dots,Z_n,W]}(\subset{R[Z_1,\dots,Z_n,W]})$.

\medskip
\indent
If $a\notin{\m}$, then $a$ is a unit in $R$ and hence $F$ is a coordinate in 
$R[ Z_{1},\dots, Z_{n},W]$. So, we assume that $a\in{\m}$. Then $\eta(F)=\eta(P)$. Since the image of $P$ is a coordinate in $\dfrac{R}{aR}[Z_1,\dots,Z_n]$, hence $g:=\eta(P)(=\eta(F))$ is a coordinate in $k[Z_1,\dots,Z_n]$ and hence in $R[Z_1,\dots,Z_n]$. Thus, there exist $g_2,\dots,g_n\in{k[Z_1,\dots,Z_n]}$ such that $k[Z_1,\dots,Z_n]=k[g,g_2,\dots,g_n]$ and hence $R[Z_1,\dots,Z_n]=R[g,g_2,\dots,g_n]$.

\medskip
\indent
Set  $A:=R[g_2,\dots,g_n]$ and $B:=R[Z_{1},\dots, Z_{n},W]=A[g,W](=A^{[2]})$. We now show that $F$ is a residual coordinate in $B=A[g,W]$. Let $Q$ be an arbitrary prime ideal of $A$ and $\p=Q\cap{R}$. If $\p=\m$, then $k=k(\p)\hookrightarrow{k(Q)}$ and the image of $F$ in $B\otimes_{A}k(Q)$, being the image of $g(=\eta(F))$ in $B\otimes_{A}k(Q)$, is a coordinate in $B\otimes_{A}k(Q)$. If $\p\neq{\m}$, then $\p$ is a minimal prime ideal of the one-dimensional ring $R$ and hence $\p\in{Ass(R)}$. Therefore, $a\notin{\p}$. Hence $a$ is a unit in $k(\p)$ and therefore the image of $F$ in $B\otimes_{A}k(Q)$ is a coordinate in $B\otimes_{A}k(Q)$. Thus, $F$ is a residual coordinate in $A[g,W]$. Hence, by Theorem \ref{res-stable}, $F$ is a coordinate in $A[g,W](=R[Z_1,\dots,Z_n,W]$).
\end{proof}

Now, we extend Theorem \ref{Kahoui} to any algebraically closed field of arbitrary characteristic. First, we prove an easy lemma on existence of exponential maps.
\noindent 
\begin{lem}\label{exp}
Let $R$ be a ring and $A=R[Z_1,\dots,Z_n]$. Let $S$ be a multiplicatively closed subset of $R$ consisting of non-zerodivisors in $R$ and $P\in{R[Z_1,\dots, Z_n]}~(=R^{[n]})$. If $P$ is a coordinate in $S^{-1}A$ then there exists an $R$-linear exponential map ${\varphi}_{W}: A\rightarrow{A[W]}$ such that $\varphi_W(P)=aW+P$, for some $a\in{S}$.  
\end{lem}

\begin{proof}
Since $P$ is a coordinate in $S^{-1}A$, there exist $c\in{S}$ and $g_2, \dots, g_n\in{A}$ such that 
$$
A_{c}= R_{c}[g_2, \dots, g_n,P]~(={R_{c}[g_2,\dots,g_n]}^{[1]}),
$$
where $A_c$ and $R_c$ denote the localisation of the rings $A$ and $R$ respectively at the multiplicative set $\{1,c,c^2,\dots\}$.
Let $Z_i=\sum{a_{i, j_1,\dots, j_n}P^{j_1}g_{2}^{j_2}\dots g_{n}^{j_n}}$, where  $a_{i,j_1,\dots, j_n}\in{R_{c}}$, for all $i, 1\leq{i}\leq{n}$
and $j_1,\dots, j_n\geq{0}$ .
Then there exists $m\geq{0}$ such that $c^{m}a_{i,j_1,\dots, j_n}\in{R}$, for all $i, 1\leq{i}\leq{n}$
and $j_1,\dots, j_n\geq{0}$. Define an $R_{c}[g_2,\dots,g_n]$-algebra homomorphism  ${\psi_{W}}: A_{c}\rightarrow{ A_{c}[W]}$ by setting ${\psi_{W}}(P):=P+c^{m}W$. Clearly, ${\psi_{W}}$ is an exponential map and ${\psi_{W}}(Z_{i})=Z_{i}+ h_i$, for some $h_i\in{A[W]}$. Now, if we set $a$ to be $c^m$ and define $\varphi_{W}:={\psi_{W}}|_A$ then $\varphi_{W}$ is our desired exponential map. 
\end{proof}
We now generalize Theorem \ref{Kahoui}.
\begin{thm}\label{ext-Kahoui}
Let $k$ be an algebraically closed field, $R$ a one-dimensional affine $k$-algebra such that either the characteristic of $k$ is zero or $R_{red}$ is seminormal. Then, every residual coordinate in $A:=R[Z_1,\dots,Z_n]~(=R^{[n]}), n\geq{3}$, is a $1$-stable 
coordinate.
\end{thm}

\begin{proof}
By Lemma \ref{nilradical}, it is enough to consider the case when $R$ is a reduced ring. Let $P(Z_1,\dots, Z_n)$ be a residual coordinate in $A$ and $S$ be the set of all non-zerodivisors in $R$. Since $S^{-1}R$ is Artinian, hence by Proposition \ref{Artinian}, $P$ is a coordinate in $S^{-1}A$. Therefore, by Lemma \ref{exp}, there exists an $R$-linear exponential map ${\varphi}_{W}: A\rightarrow{A[W]}$ such that $\varphi_W(P)=aW+P$, for some $a\in{S}$. Now, by Theorem \ref{ext-Maubach}, $aW+P$ is a coordinate in $A[W]$. Since by Proposition \ref{exp-auto}, the extension of ${\varphi}_{W}$ to $A[W]$ is an $R$-automorphism of $A[W]$ which maps $P$ to $aW+P$, $P$ is a $1$-stable coordinate in $A$.
\end{proof}
\begin{rem}
{\em Recall that if $R$ is as in Theorem \ref{ext-Kahoui} then a residual coordinate in $R[Z_1,Z_2]~(=R^{[2]})$ is actually a coordinate (\cite[Theorem 3.2]{BD1}).}   
\end{rem}
Now, using Lemma \ref{exp}, we will show that the condition ``$R$ contains $\bQ$'' can be dropped from Theorem \ref{Kahoui-general}. 
\begin{thm}\label{ext kahoui general}
Let $R$ be a Noetherian $d$-dimensional ring. Then every residual coordinate in $R[Z_1,\dots,Z_n]~(=R^{[n]})$ is a $(2^d-1)n$-stable coordinate. 
\end{thm}
\begin{proof}
We prove the result by induction on $d$. If $d=0$, the result follows from Proposition \ref{Artinian}. Now, let $d\geq{1}$ and $P$ a residual coordinate in $A:={R[Z_1,\dots,Z_n]}$. We show that $P$ is a $(2^d-1)n$-stable coordinate in $A$.\\
\indent
By Lemma \ref{nilradical}, we may assume that $R$ is a reduced ring. Let $S$ be the set of all non-zerodivisors of $R$. Then, as $S^{-1}R$ is an Artinian ring, by the case $d=0$, $P$ is a coordinate in $S^{-1}A$. Hence, by Lemma \ref{exp}, there exists a non-zerodivisor $a$ in $R$ and an exponential map ${\varphi}_{W}:A\longrightarrow{A[W]}$ such that $\varphi_{W}(P)=aW+P$. Now, we observe that the image of $P$ is a residual coordinate in $\dfrac{R}{aR}[Z_1,\dots,Z_n]$ and $\dfrac{R}{aR}$ is a $(d-1)$-dimensional ring. So, by induction hypothesis, $P$ is a $(2^{d-1}-1)n$-stable coordinate in $\dfrac{R}{aR}[Z_1,\dots,Z_n]$. Hence, by Theorem \ref{essen stable}, $aW+P$ is an $r$-stable coordinate in $A[W]$, where $r=2n(2^{d-1}-1)+n-1=(2^d-1)n-1$. Since by Proposition \ref{exp-auto}, the extension of ${\varphi}_{W}$ to $A[W]$ is an $R$-automorphism of $A[W]$ which maps $P$ to $aW+P$, $P$ is an $r+1~(=(2^d-1)n)$-stable coordinate in $A$.
\end{proof}

Next, using  Theorem \ref{ext-Maubach} and Lemma \ref{exp}, we will show that under the additional hypothesis that $R$ is affine over an algebraically closed field of characteristic zero, we can get a sharper bound in Theorem \ref{ext kahoui general}. 
\begin{thm}\label{efficient bound}
Let $k$ be an algebraically closed field of characteristic zero and $R$ a finitely generated $k$-algebra of dimension $d$. Then every residual coordinate in $R[Z_1,\dots,Z_n]~(=R^{[n]})$ is an $r$-stable coordinate, where $r=(2^d-1)n-2^{d-1}(n-1)=2^{d-1}{(n+1)}-n$.
\end{thm}
\begin{proof}
We prove the result by induction on $d$. If $d=1$, the result follows from Theorem \ref{ext-Maubach}. Let $d\geq{2}$ and $P$ a residual coordinate in $A:={R[Z_1,\dots,Z_n]}$. We show that $P$ is a $(2^{d-1}{(n+1)}-n)$-stable coordinate in $A$.\\
\indent
By Lemma \ref{nilradical}, we may assume that $R$ is a reduced ring. Let $S$ be the set of all non-zerodivisors of $R$. Since $S^{-1}R$ is an Artinian ring, by Proposition \ref{Artinian}, $P$ is a coordinate in $S^{-1}A$. Hence, by Lemma \ref{exp}, there exists a non-zerodivisor $a$ in $R$ and an exponential map ${\varphi}_{W}:A\longrightarrow{A[W]}$ such that $\varphi_{W}(P)= aW+P$. Now, we observe that the image of $P$ is a residual coordinate in $\dfrac{R}{aR}[Z_1,\dots,Z_n]$ and $\dfrac{R}{aR}$ is a $(d-1)$-dimensional ring containing an algebraically closed field of characteristic zero. So, by induction hypothesis, $P$ is a $(2^{d-2}(n+1)-n)$-stable coordinate in $\dfrac{R}{aR}[Z_1,\dots,Z_n]$. Now, arguing as in Theorem \ref{ext kahoui general}, the result follows from Theorem \ref{essen stable} and Proposition \ref{exp-auto}. 
\end{proof}

\medskip
The following question asks whether Theorem \ref{ext-Maubach} (and thereby Theorem \ref{ext-Kahoui}) can be extended to an affine algebra of \textit{any} dimension over \textit{any} field (not necessarily algebraically closed) of arbitrary characteristic.
\begin{ques}
Let $k$ be a field, $R$ an affine $k$-algebra, $a$ be a non-zerodivisor in $R$ and $P(Z_{1},\dots, Z_{n})\in{R[Z_{1},\dots, Z_{n}]}$ be such that the image of $P$ is a coordinate in $\dfrac{R}{aR}[Z_{1},\dots, Z_{n}]$. Suppose, $R_{red}$ is seminormal or the characteristic of $k$ is zero. Then, is $F:=aW+P$ a coordinate in $R[Z_{1},\dots, Z_{n},W]$?
\end{ques} 

\noindent
The following result shows that for $n=2$, the above question has an affirmative answer when $R$ is a Dedekind domain containing a field of characteristic zero.
\begin{prop}
Let $R$ be a Dedekind domain containing $\bQ$, $a\in{R-{\lbrace0\rbrace}}$ and $F= aW+P(Y,Z)\in{R[Y,Z,W]}~(=R^{[3]})$. If the image of $P$ is a coordinate in $\dfrac{R}{aR}[Y,Z]$, then $F$ is a coordinate in $R[Y,Z,W]$.
\end{prop}
\begin{proof}
By Lemma \ref{one dimensional cancellation}, it is enough to assume that $R$ is a discrete valuation ring with parameter $t$. Let $k=\dfrac{R}{tR}$, $K=R[\frac{1}{t}]$,  $A=\dfrac{R[Y,Z,W]}{(F)}$ and $\overline{P}$ denote the image of $P$ in $k[Y,Z]$. Note that $a$ is a unit in $K$ and hence $F$ is a coordinate in $K[Y,Z,W]$; in particular $A[\frac{1}{t}]=K^{[2]}$.

\medskip
\indent
 If $a\notin{tR}$, then $a$ is a unit in $R$ and hence $F$ is a coordinate in $R[Y,Z,W]$. We now consider the case $a\in{tR}$. Considering the natural map $\dfrac{R}{aR}\rightarrow{\dfrac{R}{tR}(=k)}$, we see that $\overline{P}$ is a coordinate in $k[Y,Z]$ and hence $\dfrac{A}{tA}=k^{[2]}$. It also follows that $P\notin{tR[Y,Z]}$ and hence $a$ and $P$ are coprime in $R[Y,Z]$. Thus, $F(=aW+P)$ is irreducible in $R[Y,Z,W]$. So, $A$ is a torsion free module over the discrete valuation ring $R$ and hence $A$ is flat over $R$. Thus, $A$ is an $\A^{2}$-fibration over $R$ and hence by Theorem \ref{Sathaye}, $A=R^{[2]}$.
 As $\overline{P}$ is a coordinate in $k[Y,Z]$, we see that $\overline{P}\notin{k[Y]\cap{k[Z]}}(=k)$. Hence at least one of the rings $k[Y,\overline{P}]$ and $k[Z,\overline{P}]$ is of dimension two. Now, the result follows from Theorem \ref{dvr}.  
\end{proof}

\bigskip
\noindent
{\bf Acknowledgements.} The authors thank S. M. Bhatwadekar and Neena Gupta for the current version of Theorem \ref{ext-Maubach} and Theorem \ref{efficient bound} and their useful suggestions. The second author also acknowledges Council of Scientific and Industrial Research (CSIR) for their research grant.



{\small{

}}
 \end{document}